\documentclass{amsart}
\usepackage{graphicx, amssymb}
\usepackage[alphabetic,y2k,lite]{amsrefs}
\usepackage{lscape}

\newtheorem*{theorem*}{Theorem}
\newtheorem{theorem}{Theorem}[section]
\newtheorem{lemma}[theorem]{Lemma}

\theoremstyle{definition}
\newtheorem{definition}[theorem]{Definition}

\newtheorem{example}[theorem]{Example}

\theoremstyle{remark}

\numberwithin{equation}{section}
\newcommand{\firef}[1]{Figure~{\rm\ref{#1}}}
\newcommand{\thref}[1]{Theorem~{\rm\ref{#1}}}

\newcommand{\leref}[1]{Lemma~{\rm\ref{#1}}}

\newcommand{\deref}[1]{Definition~{\rm\ref{#1}}}

\newcommand{\seref}[1]{Section~{\rm\ref{#1}}}

\newcommand{\fig}[1]
{\raisebox{-0.5\height}%
{\includegraphics{#1}}}
\newcommand{\figscale}[2]
{\raisebox{-0.5\height}%
{\includegraphics[scale=#1]{#2}}}

\newcommand{\<}{\langle}
\renewcommand{\>}{\rangle}

\newcommand{\one}{\mathbf{1}}
\newcommand{\DD}{\mathcal{D}}      
\newcommand{\C}{\mathcal{C}}      
\newcommand{\M}{\mathcal{M}}      

\DeclareMathOperator{\Irr}{Irr}

\DeclareMathOperator{\Hom}{Hom}

\DeclareMathOperator{\Dim}{Dim}
\begin{document}
\title{Turaev-Viro invariants as an extended TQFT III}

\author{Benjamin Balsam}
   \address{Department of Mathematics, SUNY at Stony Brook, 
            Stony Brook, NY 11794, USA}
    \email{balsam@math.sunysb.edu}
    \urladdr{http://www.math.sunysb.edu/\textasciitilde balsam/}
    \thanks{This  work was partially suported by NSF grant DMS-0700589 }
    
\begin{abstract}
In the third paper in this series, we examine the Reshetikhin-Turaev and Turaev-Viro TQFTs at the level of surfaces. In particular, we show that for a closed surface $\Sigma$, $Z_{TV, \mathcal{C}}(\Sigma) \cong Z_{RT, Z(\C)}(\Sigma)$, thus extending the equality of 3-manifold invariants proved in \ocite{mine2} to an equivalence of TQFTs. We also describe how to compute Turaev-Viro state sums on 3-manifolds with embedded ribbon graphs.
\end{abstract}
\maketitle
\section*{Introduction}
In this paper we continue the work from \ocite{mine}, \ocite{mine2} in which we generalized the Turaev-Viro state-sum invariant to manifolds with corners. This gave an extended Topological Quantum Field Theory (TQFT). Using this extended theory, we showed that for a closed 3-manifold $\M$, $Z_{TV, \C}(\M) = Z_{RT, Z(\C)}(\M)$, where $\C$ is a spherical fusion category, $Z(\C)$ is its Drinfeld Center (which is modular) and $Z_{RT, Z(\C)}$ is the Reshetikhin-Turaev invariant based on $Z(\C)$. 
In this paper, we show that the TQFTs are isomorphic at the level of surfaces. Namely, if $\Sigma$ is a closed surface, we show that there is a nautural isomorphism $Z_{TV, \mathcal{C}}(\Sigma) \cong Z_{RT, Z(\C)}(\Sigma)$ of vector spaces. We also note that we actually get an equivalence of extended 3-2-1 theories if we impose mild restrictions on the allowed types of manifolds with corners.

It is easy to compute the dimensions of the above spaces: 
\begin{equation}
\Dim Z_{TV}(\Sigma_g) = \Dim Z_{RT}(\Sigma_g) = \DD^{2g-2}\displaystyle \sum_{i \in \Irr(\C)}d_{i}^{2-2g}
\end{equation}
where $\DD$ is the dimension of $\C$ and $d_i$ is the dimension of simple object $X_i$. The vector spaces are therefore isomorphic, but this is not enough. We need to exhibit a natural isomorphism between the spaces.

The same question occurs in general when attempting to define any 2D modular functor. For example, in RT theory, one decomposes the surface $\Sigma$ into a union of punctured spheres\footnote{Following \ocite{BK2}, we call this a \textit{cut sytem}.}, evaluates $Z_{RT}$ for each of them, and uses the gluing axiom to obtain $Z_{RT}(\Sigma)$. A priori, this appears to depend on the choice of decomposition of $\Sigma$. Refining earlier work by Moore and Seiberg, Bakalov and Kirillov \ocite{BK2} proposed a set of \textit{moves} (The "Lego-Teichm\"uller Game"). relating any two such decompositions. One can show that each of these moves  corresponds to a certain natural isomorphism of vector spaces and any two "paths" between two chosen decompositions yield the same map. The space $Z(\Sigma)$ is therefore well defined.

In this paper we apply the results described above to TV theory. In  \ocite{mine}, we constructed an isomorphism  
\begin{equation}
 Z_{TV}(\Sigma) \cong \Hom_{Z(\C)}(\one, Y_1 \otimes \dots \otimes Y_n)
 \end{equation}
where $\Sigma$ is an n-punctured sphere with boundary components labeled by $Y_1, \dots Y_n \in Irr(Z(\C))$. Notice that the space on the right of this equation is by definition  $Z_{RT, Z(\C)}(\Sigma; Y_1, \dots Y_n)$.

 It is important to note that RT is defined using "pairs-of-pants" decompositions  of surfaces, while TV is defined via cell decompositions. Since the latter is a local construction and the former is inherently nonlocal, comparing the two requires a natural way of passing between them. The solution is simple and is provided immediately by the surface parametrizations defined in [BK]. Using these, we can compute maps between TV state spaces that correspond to each of the moves between cut systems and check that such maps are compatible with the projector $H_{TV}(\Sigma) \longrightarrow Z_{TV}(\Sigma)$. Thus, we get a natural identification $Z_{RT}(\Sigma) \cong Z_{TV}(\Sigma)$.
 
 In both RT and TV theory, once we know the value of the TQFT on a punctured sphere, we can use the gluing axiom to define $Z(\Sigma)$ for any surface. Thus, we get a well-defined vector space, up to natural isomorphism that depends only on the topology of $\Sigma$. We do this in each case by defining "intermediate" vector spaces which do depend on some choices\footnote{The parametrization in RT and the cell decomposition in TV.} and demonstrating that we can identify all such spaces naturally. The key result in this paper is that we can pass between the theories in a natural way, so that $Z_{TV, \C}(\Sigma) \cong Z_{RT, Z(\C)}(\Sigma)$ independent of any choices. 

The paper is organized as follows. First, we briefly review the theory of parametrized surfaces from \ocite{BK2}.  Next, we examine the effect of passing between parametrizations on the associated TV state spaces. In particular, we show that each of the moves yields a natural map between state spaces, which under projection gives the same identification between vector spaces as that in RT. This establishes an equivalence of theories at the level of surfaces. Finally, we consider extended 3-manifolds with boundary and show that both theories give the same answer. Along the way, we explain how Turaev-Viro theory works for 3-manifolds with embedded ribbon graphs. The appendix contains some of the larger diagrams referenced in the paper. 

This paper completes the program outlined in \ocite{mine} and continued in \ocite{mine2}. The reader is strongly encouraged to read these papers before this one, as they contain much prerequisite material.
\subsection*{Acknowledgments}
The author would like to thank Sasha Kirillov for his help in writing this paper.
\section{Surface Decompositions}\label{s:lego}
In this section we briefly review the notion of a parametrized surfaces. For a complete exposition, see \ocite{BK2}.  Informally, a parametrization is a way of writing a surface $\Sigma$ as the union on punctured spheres, together with a fixed identification of each punctured sphere with a \textit{standard} sphere. 
The \textit{standard sphere} with n punctures is defined formally as
\begin{equation}
 S_{0, n} = \mathbb{CP}^1 \backslash \{D_1, \dots, D_n\};  D_j = \{z | |z-z_j| < \epsilon\}, z_1 < \dots < z_n 
\end{equation}
where $\epsilon$ is sufficiently small so the boundary circles do not intersect. We also fix a point $p_i \in \partial D_i$. Note that we have fixed an ordering of the boundary circles, so we can refer to the set of boundary components by \{\textbf{1}, \dots, \textbf{n}\}. 
\begin{definition}
An \textit{extended surface} is a compacted oriented surface $\Sigma$ , possibly with boundary, together with a fixed point $p_\alpha$ on each boundary component $(\partial \Sigma)_\alpha$.
\end{definition}

Note that there are several other equivalent ways of defining an extended surface (See \ocite{BK}).

\begin{definition} A \textit{colored} extended surface is an extended surface together which a choice of label $Z_{\alpha} \in Z(\mathcal{C})$ for each marked point $p_{\alpha}$.
\end{definition}
We now give the main definition of this section. Let $\Sigma$ be a colored extended surface.
\begin{definition}
A \textit{parametrization} of $\Sigma$ consists of 
\begin{enumerate}
\item A finite set $C$ of non-intersecting simple closed curves on $\Sigma$ such that $\Sigma \backslash C$ is of genus zero. We call $C$ the set of \textit{cuts} and fix a point on each cut. 
\item For each component $\Sigma_a$ of $\Sigma \backslash C$, a homeomorphism $\psi: \Sigma_a \to S_{0, n_a}$
\end{enumerate}
\end{definition}
Two parametrizations are considered equivalent if they are isotopic \footnote{Both the set of cuts and the homemorphisms of boundary components are considered up to isotopy}. 
There is a nice graphical way of describing parametrized surfaces. Namely, take the standard sphere with the graph as shown in \firef{f:paramsphere}.
\begin{figure}[ht]
\figscale{.5}{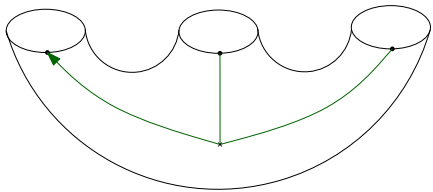}
\caption{The graph on $S_{0,3}$} \label{f:paramsphere}
\end{figure}
This graph connects a single internal vertex to each of the points $p_{\alpha}$ fixed on the boundary and labels the edge connected to circle \textbf{1} by an arrow. 

To depict a parametrization of any surface $\Sigma$, we draw the cuts on $\Sigma$. Then for each connected component $\Sigma_{\alpha}$, we pull back the graph on the standard sphere by $\psi_{\alpha}$ to obtain a graph $M_{\alpha}$ on $\Sigma_{\alpha}$. Clearly, such data are equivalent (up to isotopy) to specifying a parametrization and henceforth we will refer to a parametrization as a pair $(C, M)$ where $C$ is a set of cuts on $\Sigma$ and $M = \cup_{\alpha}M_{\alpha}$. 
\begin{figure}[ht]
\figscale{.5}{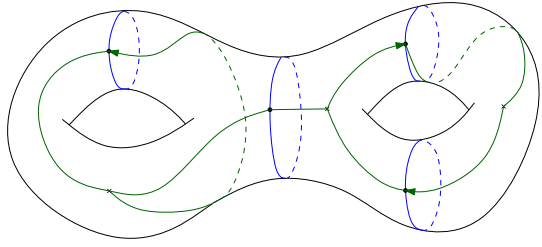}
\caption{A parametrization of a genus two surface $\Sigma = S_{0,3} \sqcup S_{0,3} \sqcup S_{0,2}$. The blue lines are cuts and the green lines are graphs $M_{\alpha}$}
\end{figure}
When possible, we will often draw the graphs $M_{\alpha}$ in the plane, ignoring the sufaces into which they are embedded. The reader should have no difficulty passing between such a graph and the surface it represents.
\begin{definition}
Let $\Sigma$ be a parametrized sphere with n boundary components colored by $Z_1, \dots, Z_n \in \Irr(Z(\mathcal{C}))$. We define the Reshetikhin-Turaev invariant of $\Sigma$ to be
\begin{equation}
Z_{RT, Z(\mathcal{C})}(\Sigma; Z_{1}, \dots, Z_{n}) = \Hom_{Z(\mathcal{C})}(\one, Z_{1} \otimes \dots \otimes Z_{n})
\end{equation}
\end{definition}
More generally, we can defined the Reshetikhin-Turaev invariant for any colored, parametrized surface as follows. $\Sigma \backslash C$ is a union of genus zero surfaces with boundary, each equipped with parametrization inherited from $\Sigma$. Let $\Sigma_{\alpha}, \Sigma_{\beta}$ be two such components separated by a cut $c$. Then $c$ corresponds to two boundary circles, one on $\Sigma_{\alpha}$ and the other on $\Sigma_{\beta}$. We may color these components by assigning $Z \in \Irr(Z(\mathcal{C}))$ to one component and $Z^*$ to the other.
\begin{definition}
Let $(\Sigma, P)$ be a colored parametrized surface.
\begin{equation}
Z_{RT, Z(\C)}(\Sigma, P) = \displaystyle \bigoplus_{Y_{1_{\alpha}}, \dots Y_{n_{\alpha}}} \bigotimes_{\alpha} Z_{RT, Z(\C)}(\Sigma_{\alpha}; Y_{1_{\alpha}}, \dots Y_{n_{\alpha}})
\end{equation}
where the product is over all connected components of $\Sigma \backslash C$, and we color all \textit{newly created}  boundary components as described above, summing over all possible colorings.
\end{definition}

We will typically denote a simple object $Y_i \in Z(\C)$ by its index $i$. In all that follows, $i^*$ represents the dual object $Y_i^*$, which is also simple. To simplify formulas, many authors attempt to pick a function $f: Irr(\C) \to Irr(\C)$ so that $Y_{i}^* = Y_{f(i)}$, but one should avoid doing this at all costs since it is often impossible to do so in a consistent manner (See \ocite{BK}, Remark 2.4.2).
\begin{example}
Let $\Sigma$ be the torus with one puncture and parametrization as shown on the left hand side of \firef{f:smove} and boundary disk labeled by $Y$. Then $Z_{RT, Z(\C)}(\Sigma) = \displaystyle \bigoplus_{i \in Irr(Z(\C))}Hom_{Z(\C)}(\one, Y \otimes i \otimes i^*)$.
\end{example}

Now we describe a set of \textit{moves} between parametrizations of a surface. As we'll see below, we can relate any two decompositions by a finite composition of these moves:

\begin{enumerate}
\item The Z-move cyclically permutes the boundary components.
\item The B-move \textit{braids} one boundary component about an adjacent one.
\item The F-move removes a cut. If a cut separates $S_{0,n}$ and $S_{0,m}$, deleting the cut gives a component homeomorphic to $S_{0,m+n-2}$ together with a graph inherited from the original components. Notice that we connect circle \textbf{1} from the one sphere to circle \textbf{m} of the other, thus resulting in a graph which inherits a natural ordering of boundary circles.
\item The S-move interchanges meridians and longitudes of the punctured torus.
\end{enumerate}
\begin{figure}
\figscale{.8}{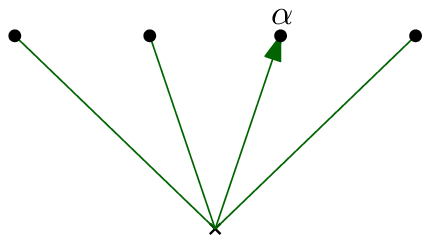} \hspace{1cm} $\stackrel{Z}{\longrightarrow}$ \hspace{1cm} \figscale{.8}{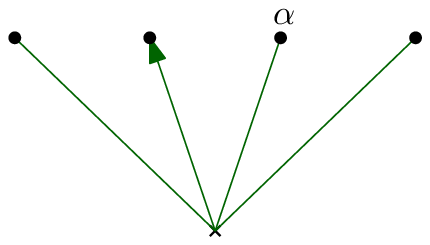}
\caption{Z-move}
\end{figure}
\begin{figure}
\figscale{.4}{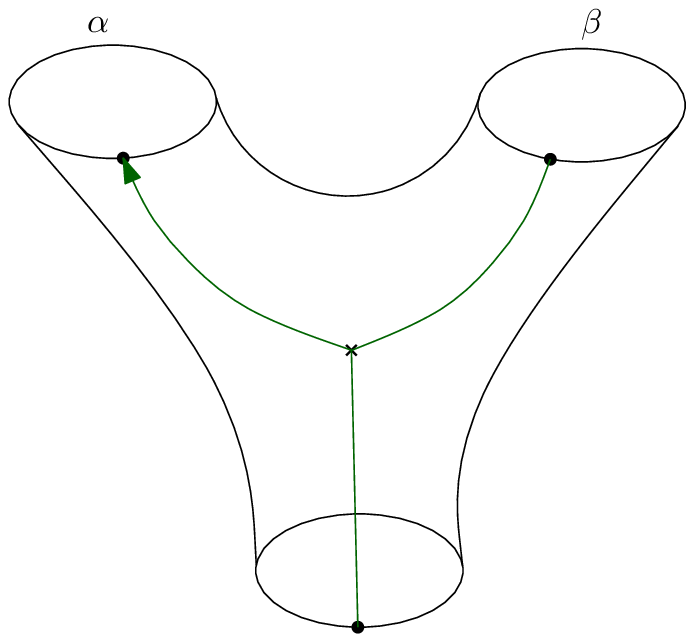} \hspace{1cm} $\stackrel{B}{\longrightarrow}$ \hspace{1cm} \figscale{.4}{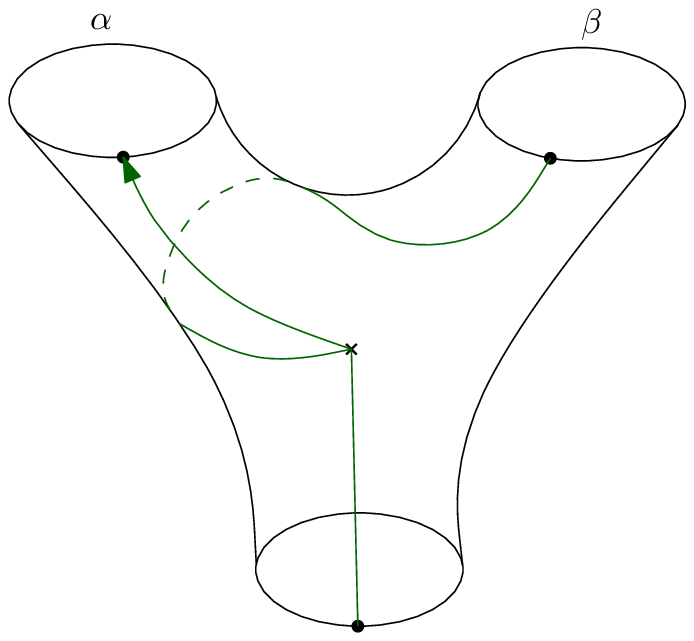}
\caption{B-move}
\end{figure}
\begin{figure}
\figscale{.5}{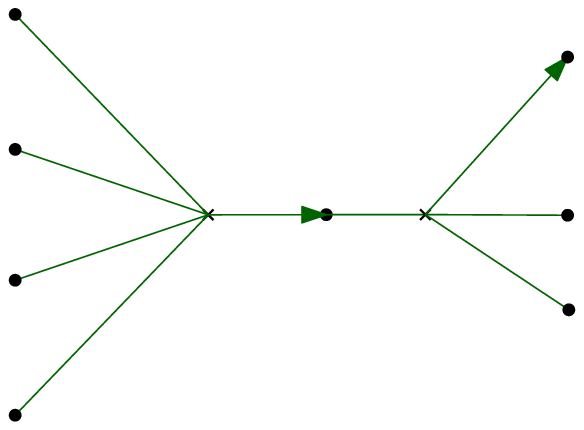} \hspace{1cm} $\stackrel{F}{\longrightarrow}$ \hspace{1cm} \figscale{.5}{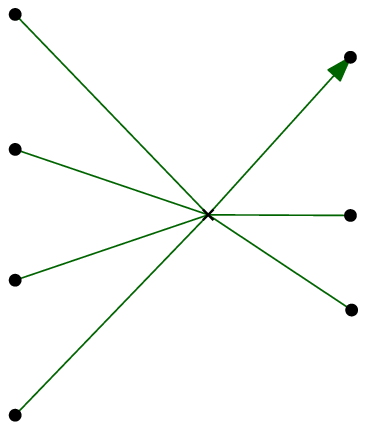}
\caption{F-move}
\end{figure}
\begin{figure}\label{f:smove}
\figscale{.4}{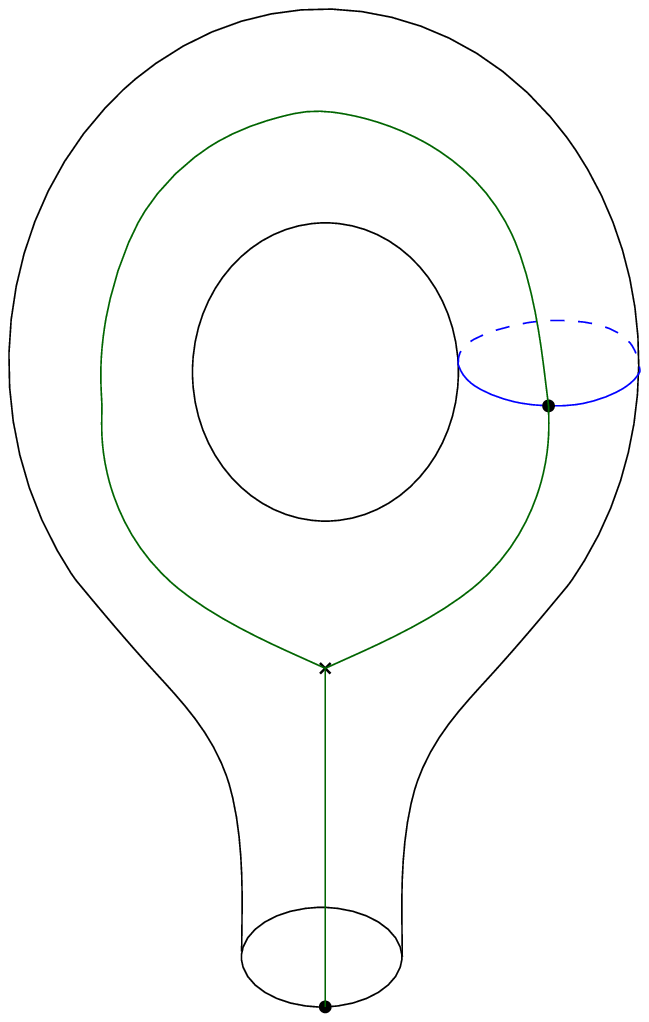} \hspace{1cm} $\stackrel{S}{\longrightarrow}$ \hspace{1cm} \figscale{.4}{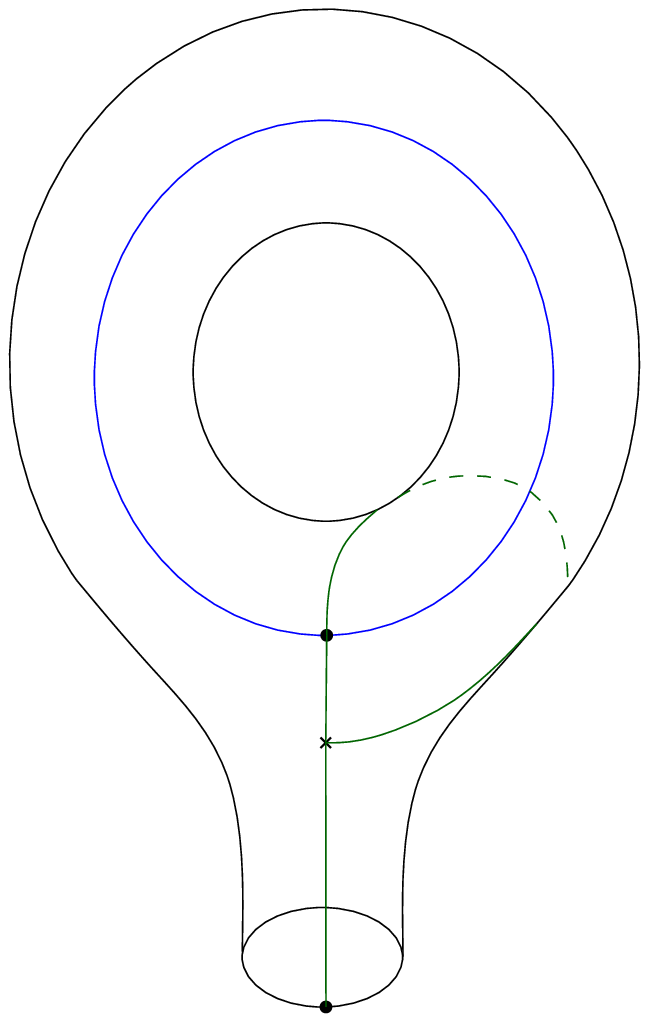}
\caption{S-move}
\end{figure}

\begin{theorem}
Let $A = (\Sigma, C, M)$ and $A' = (\Sigma, C', M')$ be two parametrizations of a surface. Then $A$ and $A'$ are related by a finite sequence of $Z,B,F$ and $S$ moves described above.
\end{theorem}
This result has its origins in conformal field theory. It was conjectured by Moore and Seiberg and rigorously proved in \ocite{BK2}. It is a generalization of a result by Hatcher and Thurston (\ocite{HT}), which describes moves between surfaces decomposed into spheres, cylinders and pairs-of-pants, but doesn't take into account the full data of a parametrization. The theorem in \ocite{BK2} does a lot more in fact: it provides a complete set of relations between the above moves,  but we will not need this part explicitely.

These moves are important for defining a 2-dimensional (extended) modular functor $\mathcal{F}$. Given the vector space associated to the punctured sphere, one should be able to use the gluing axiom to describe $\mathcal{F}(\Sigma)$ for a surface of any genus. Different parametrizations should give naturally isomorphic vector spaces; one can check that this is so by verifying that it is true for each of the \textit{simple} moves between parametrizations.
If $P, P'$ are two parametrizations of a surface $\Sigma$ related by a single $Z,B, F$ or $S$ move, we can explicitely describe the correspondence between associated vector spaces in RT theory:
\begin{lemma} \label{l:isomorphisms}
Let $P, P'$ be two parametrizations of a surface $\Sigma$ and let $X: (\Sigma, P) \longrightarrow (\Sigma, P')$ be any composition of $Z, B, F$ and $S$ moves connecting $P$ and $P'$. Then $X$ induces an isomorphism $X_{*}: Z_{RT, Z(\C)}(\Sigma, P) \stackrel{\cong}{\longrightarrow} Z_{RT, Z(\C)}(\Sigma, P')$. This isomorphism is independent of the choice of $X$.
In terms of the generators,
\begin{enumerate}
\item The Z-move corresponds to the rotation isomorphism: \\ $\<Y_1, \dots, Y_n\> \to \<Y_n, Y_1, \dots, Y_{n-1}\>$
\begin{center}
\figscale{.6}{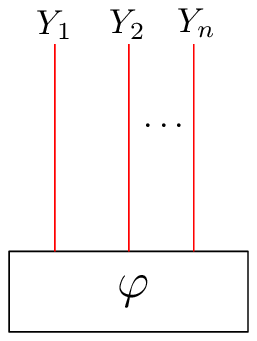} \hspace{.5cm} $\stackrel{Z_{*}}{\longrightarrow}$ \hspace{.5cm} \figscale{.6}{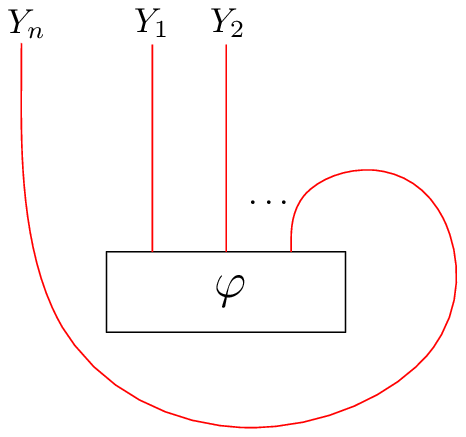}
\end{center}
\item The F-move gives the composition isomorphism. That this is an isomorphism follows directly from semisimplicity.
\begin{center}
$\displaystyle \sum_{i \in Irr(Z(\C))}$ \figscale{.6}{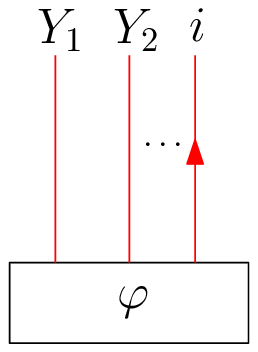} \figscale{.6}{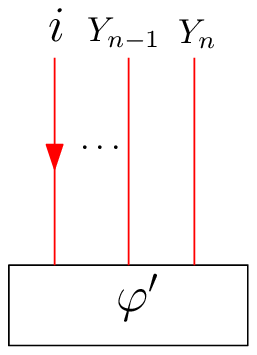} $\stackrel{F_{*}}{\longrightarrow} \displaystyle \sum_{i}$ \figscale{.6}{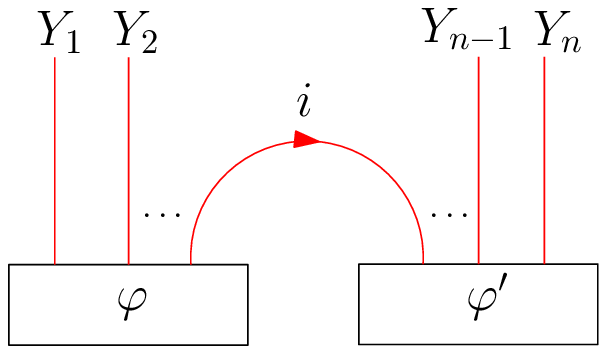}
\end{center}
\item The B-move gives the braiding isomorphism
\begin{center}
\figscale{.6}{zmovegraph1.eps} \hspace{.5cm} $\stackrel{B_{*}}{\longrightarrow}$ \hspace{.5cm} \figscale{.6}{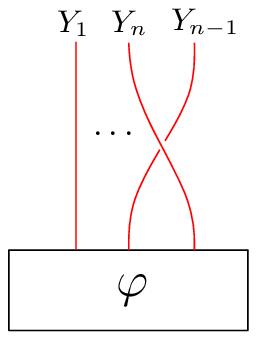}
\end{center}
\item The S-move gives multiplication by the S-matrix
\begin{center}
$\displaystyle \sum_{B} \hspace{.1cm}$\figscale{.6}{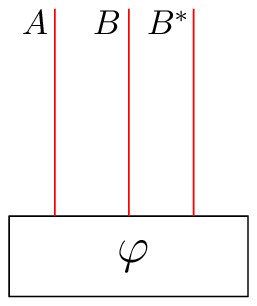} \hspace{.5cm} $\stackrel{S_{*}}{\longrightarrow}$ \hspace{.5cm} $\frac{1}{\mathcal{D}^2}\displaystyle \sum_{B, Y} \figscale{.6}{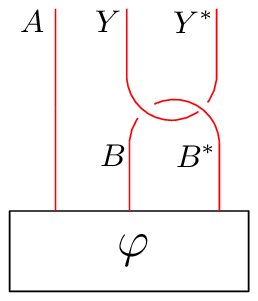}$
\end{center}
\end{enumerate}
\end{lemma}
\section{Parametrized surfaces and cell decompositions}{\label{s:mainsection}
In this section, we state and prove the main result of the paper: TV and RT theories assign the same vector space (up to natural isomorphism) surface $\Sigma$, which may have boundary. If $\partial \Sigma \neq \emptyset$, we fix a coloring of each boundary component $\partial \Sigma_i$ by $Z_i \in \Irr(Z(\mathcal{C}))$.

Given two cell decompositions $\Delta, \Delta'$ of a surface $\Sigma$, there is a natural map
\begin{equation} \label{e:natmap}
\Psi_{\Delta', \Delta}: H(\Sigma, \Delta) \longrightarrow H(\Sigma, \Delta')
\end{equation}
obtained by computing a state sum on the cylinder $\Sigma \times I$, with a decomposition chosen to agree with $\Delta$ on $\Sigma \times 0$ and $\Delta'$ on $\Sigma \times 1$. Note, that this map does not depend on the choice of the internal decomposition. We will refer to this map as the \textit{cylinder} map. \\
The cylinder map is not an isomorphism in general since the dimension of $H_{TV, \C}(\Sigma, \Delta)$ depends on the number of edges of $\Delta$, but it is almost an isomorphism. More precisely, define the space $Z_{TV, \C}(\Sigma, \Delta) = Im(\Psi_{\Delta, \Delta})$. Then 
\begin{equation} \label{e:natis}
\Psi_{\Delta, \Delta'}: Z_{TV, \C}(\Sigma, \Delta) \longrightarrow Z_{TV, \C}(\Sigma, \Delta')
\end{equation}
is a natural isomorphism. We can refer to this space as $Z_{TV, \C}(\Sigma)$, since up to natural isomorphism it doesn't depend on the cell decomposition.

Given a parametrized surface $\Sigma$, there is a natural way to obtain a cell decomposition of $\Sigma$ . We have a fixed collection of closed curves dividing $\Sigma$ into the union of punctured spheres. These cuts become 1-cells in the cell decomposition. Further, for each punctured sphere thus obtained, we have a graph from our parametrization terminating at fixed points on the boundary circles. Each edge of this graph becomes a 1-cell and the points at which the 1-cells terminate become vertices. It is easy to see that these choices define a cell decomposition in the sense of \ocite{mine} \footnote{See, in particular, Figure 27}. We call the cell decomposition obtained in this way, the \textit{associated} cell decomposition to parametrization $P$.

Recall that for a punctured sphere with standard cell decomposition (\firef{f:paramsphere}), we have a projection $H_{TV, \C}(S^2) \stackrel{\pi}{\rightarrow} Z_{TV, \C}(S^2) \cong Z_{RT, Z(\C)}(S^2)$. The associated inclusion map $i$ can be described graphically:
\begin{figure}[ht] 
$\varphi$ $\hspace{.2cm} \stackrel{i}{\longrightarrow}$ \hspace{.2cm} $\displaystyle \bigoplus_{x_1, \dots , x_n} \prod_{j=1}^{n}\sqrt{d_j} \hspace{.2cm}$\figscale{1.25}{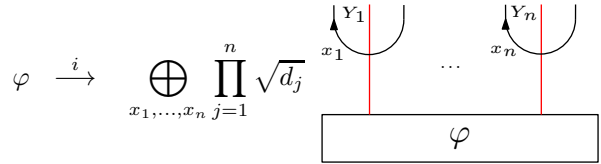}
\caption{$i: Z_{RT, Z(\C)}(S^2) \hookrightarrow H_{TV, \C}(S^2)$} \label{f:inclusion}
\end{figure}
The normalization factors are chosen to agree with that in  \ocite{mine}, so that $\pi \circ i = Id$.

We have two parallel notions in TV and RT theory. On the RT side, we have surface parametrizations and passing between any two parametrizations gives an isomorphism as described earlier in \leref{l:isomorphisms}. On the TV, side, we have cell decompositions; passing between any two cell decompositions gives a natural isomorphism obtained from a cylinder as described earlier. The following theorem, which implies the main result in this paper, shows that these two notions are the same, up to projection. 
\begin{theorem} \label{t:main}
Let $P, P'$ be two parametrizations of a surface $\Sigma$ with associated cell decompositions $\Delta, \Delta'$ respectively. Then the diagram in \firef{f:diagram} commutes.

\begin{figure} 
\figscale{.8}{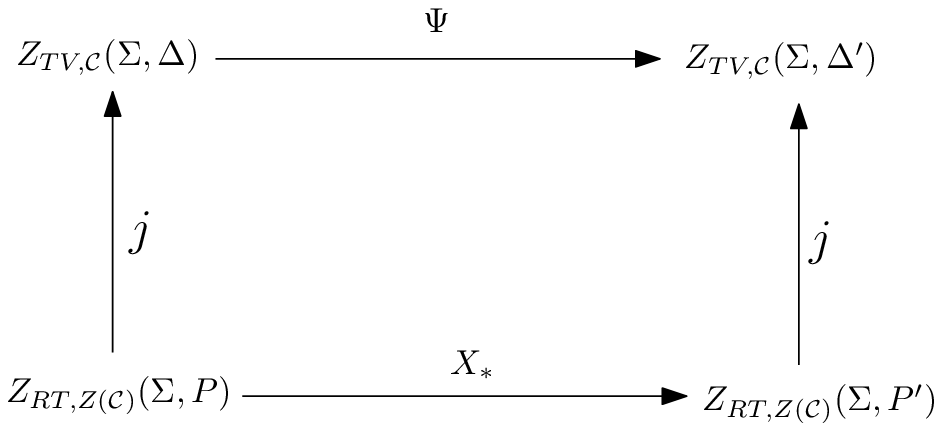}
\caption{} \label{f:diagram}
\end{figure}
Here, $X_{*}$ is the map described in \leref{l:isomorphisms}, $j$ is the map described in \firef{f:inclusion} followed by projection to $Z_{TV,\C}(\Sigma)$and $\Psi$ is the isomorphism described in \eqref{e:natis}.
\end{theorem}
\begin{proof}
To show the diagram commutes, we will verify that it does for each of the generators $Z, B, F$ and $S$. The $Z$ and $B$ moves are essentially immediate, while the $F$ and $S$ moves require some work. Throughout the proof, our convention will be that diagrams of surfaces represent the vector spaces associated to them. In particular, a parametrized surface $(\Sigma,P)$ represents $Z_{RT, Z(\C)}(\Sigma, P)$ and a cell-decomposed surface $(\Sigma, \Delta)$ represents $Z_{TV, \C}(\Sigma, \Delta)$. In the diagrams below, we have written $\Psi$ from \firef{f:diagram} as the composition of several elementary steps for the reader's edification. We have moved several of the large diagrams to \seref{s:appendix}.
\subsection*{The Z-move}
This follows directly from the natural isomorphism from \leref{l:isomorphisms}(1). 
\subsection*{The B-move}
A proof of this fact may be found in \ocite{mine} (lemma 2.1), where we provide an explicit computation.
\subsection*{The F-move}
We will show that the diagram in \firef{f:fusion} commutes . The arrow labeled $F$ is the isomorphism described in \leref{l:isomorphisms}, those labeled $i$ are inclusion maps (\firef{f:inclusion}), and $G$ is the gluing isomorphism at the level of state-spaces (Theorem 7.3, \ocite{mine}) . The other maps are all \textit{cylinder} maps \ref{e:natmap}. Notice that this can be done in fewer steps, but the cylinder maps will be more difficult to realize. \firef{f:fusion} by contrast contains cylinder maps that are all easy to compute.

To check that this diagram commutes we begin with a vector in $Z_{RT, Z(\C)}$ and proceed about the diagram in two ways. In \firef{f:fvect} we give the answer. The explicit computation at each stage left to the reader.
\subsection*{The S-move}
We will show that the diagram in \firef{f:sdiagram} commutes. We have omitted some intermediate steps on the right side of the diagram as they are much the same as those on the left. Notice that the diagrams connected by the horizontal arrow labeled $S$ are parametrized surface while the others are of cell-decomposed surfaces. We have chosen a convenient cell decomposition as the terminating point of the diagram  which is easy to work with since there are simple maps $\alpha, \beta$ to this space which can be though of as contractions along edges $u_1$ and $u_2$ respectively (\firef{f:smovemap}).
\begin{figure}[ht] 
\figscale{.8}{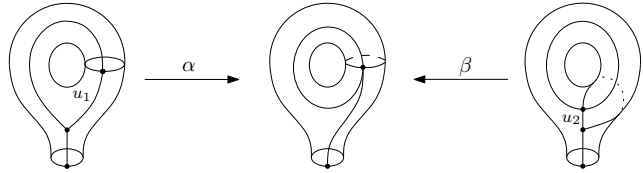}
\caption{To identify the spaces on the left and the right, we use cylinder maps $\alpha, \beta$ to the space in the center and compare the images of these maps.}\label{f:smovemap}
\end{figure}
If we start on the bottom left of figure \firef{f:sdiagram} and proceed around in two different ways, we get two vectors, $\varphi_1, \varphi_2$ in the same space  as shown in \firef{f:compare}

\begin{figure}[ht] 
$\varphi_1 = \frac{1}{\mathcal{D}}\displaystyle \sum_{x_1, x_2, N} (d_{x_1}d_{x_2}d_N)^{\frac{1}{2}} \hspace{.1cm} \figscale{.5}{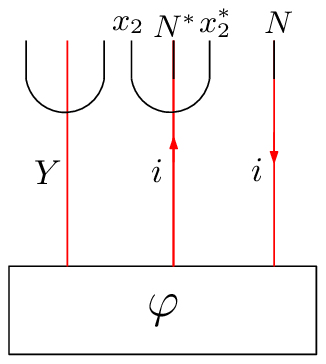} \hspace{.3cm} ;  \varphi_2 = \frac{1}{\mathcal{D}^3}\displaystyle \sum_{x_1, x_2, N} (d_{x_1}d_{x_2}d_N)^{\frac{1}{2}} \hspace{.1cm} \figscale{.5}{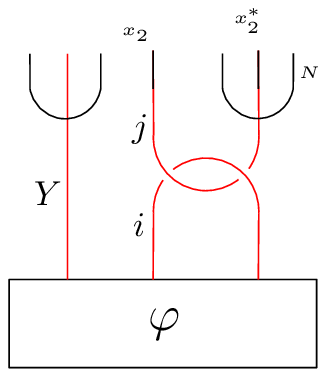}$
\caption{}\label{f:compare}
\end{figure}
We can easily verify that these vectors are the same by picking some vector $w$ in the dual space and comparing the pairings $\<\varphi_1, w\>$ and $\<\varphi_1, w\>$. 
Let $w$ be given by
\begin{center}
$\displaystyle \frac{1}{\mathcal{D}}\sum_{i,N, x_1, x_2}(d_Nd_{x_1}d_{x_2})^{\frac{1}{2}}\figscale{.8}{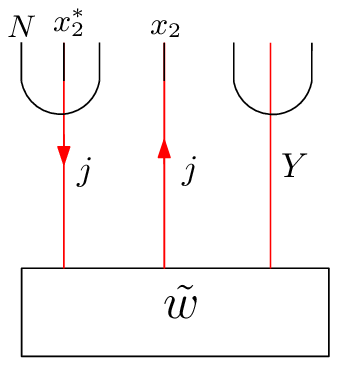}$
\end{center}
where $\tilde{w}$ is some vector in $\displaystyle \bigoplus_{j} Hom_{Z(\C)}(\one, j^* \otimes j \otimes Y)$. Then
\begin{center}
$\<w, \varphi_1\> = \displaystyle \frac{1}{\mathcal{D}^2}\sum d_N d_{x_1} d_{x_2}\figscale{.5}{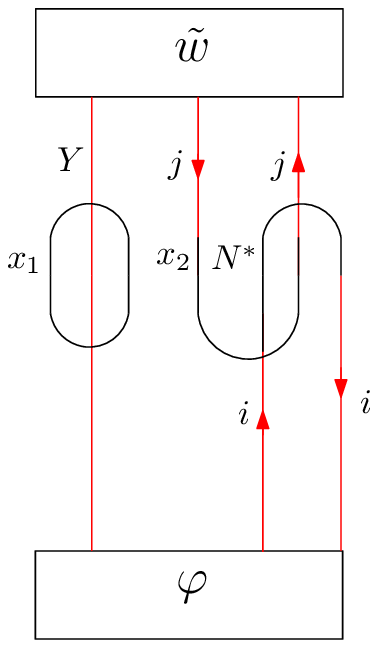} = \figscale{.5}{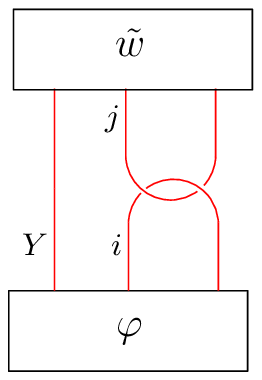}
 = \<\tilde{w}, S_{*}\varphi\>$
\end{center}
$ = \displaystyle \frac{1}{\mathcal{D}^4}\sum d_N d_{x_1} d_{x_2}\figscale{.5}{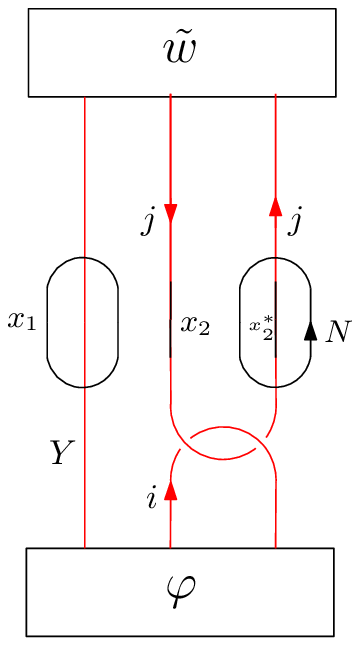} = \<w, \varphi_2\>$
\end{proof}
As an immediate consequence, we get
\begin{theorem} \label{t:main2}
For any  surface (possibly with boundary), we have a natural isomorphism 

$Z_{RT, Z(\C)}(\Sigma) \cong Z_{TV, \C}(\Sigma)$.
\end{theorem}

\section{Equivalence of Extended Theories}
We conclude the paper by combining results in \ocite{mine}, \ocite{mine2} and this paper to prove the following theorem.
\begin{theorem} \label{t:mainone}
Let $\mathcal{C}$ be a spherical fusion category. Then $Z_{TV, \mathcal{C}} \cong Z_{RT, Z(\mathcal{C})}$ as (3-2-1) TQFTs.
\end{theorem}

We have already shown that the two TQFTs give the same answer for a closed 3-manifold with an embedded link \ocite{mine} (Theorem 4.8), and for a surface, possibly with boundary, as shown in \thref{t:main2}. It remains to show that the theories give the same answer on any 3-manifold with corners. 

Let us very briefly review the RT construction for 3-manifolds with corners. For more details, see \ocite{BK}, \ocite{turaev}. Fix a spherical fusion category $\mathcal{C}$.

\begin{definition}
An extended 3-manifold $\mathcal{M}$ is an oriented PL 3-manifold with boundary, together with a finite collection of disjoint framed tubes $T_{i} \subset \mathcal{M}$.
\end{definition}

An extended 3-manifold as described above is equivalent to a 3-manifold with an embedded framed tangle in the obvious way. We will use both descriptions interchangeably.

Notice that a tube $T_i$ may terminate on $\partial \mathcal{M}$ in which case we call it an \textit{open} tube, or it may close on itself, forming a solid torus, in which case we call it a \textit{closed} tube.

\begin{definition}
 A coloring of an extended 3-manifold $\mathcal{M}$ is a choice of color of simple object $Y \in Z(\mathcal{C})$ for each open tube $T_i$.
\end{definition}

We wish now to generalize the famous theorem which states that any closed 3-manifold may be obtained from $S^3$ via surgery along a framed link.

\begin{definition}
A framed link with \textit{coupons} is a framed link where components are allowed to coincide at multivalent vertices, called \textit{coupons}. We often draw coupons as rectangles instead of vertices (\firef{f:coupon}).
\end{definition}

Given $L \subset \mathbb{R}^3$, an oriented, framed link with coupons, we can color $L$ as follows. As before, assign to each edge, a simple object $Y \in Z(\mathcal{C})$. To each coupon $C_i$ assign a morphism $\varphi_{i} \in W_{i} \equiv \Hom_{Z(\mathcal{C})}(\one, Z^{\epsilon_1}_{1} \otimes \dots \otimes Z^{\epsilon_n}_{n})$, where $Z_1 \dots Z_n$ are the colors of the edges incident the coupon in clockwise cylic order, and $\epsilon_i = 1$ if the strand labeled by $Z_i$ is oriented away from the coupon and $-1$ otherwise.

As shown in \ocite{turaev}, we can evaluate such a link $L \subset \mathbb{R}^3$ to get a number $Z_{RT}(L) \in \mathbb{C}$ in a way that is invariant under isotopy of $L$. Further, since $\mathcal{C}$ (and hence $Z(\mathcal{C}$)) is a spherical category, we can actually view $L$ as lying in $S^3$.

Equivalently, if we leave the coupons of $L$ uncolored, this construction gives a vector $v \in \displaystyle \bigotimes_{i} W^{*}_{i}$, where the tensor product is over all coupons in $L$ and $W_i$ is the $\Hom$-space associated to coupon $C_i$. Thus, such a link with uncolored coupons gives a vector space $V$ and a vector $v_{L} \in V$. Both can be seen to be invariant under isotopy of $L \in S^3$.

We are most interested in a particular type of oriented, framed link with coupons:
\begin{definition}
A \textit{special} link $X$ is a framed link with coupons such that some of the link components and coupons are colored by objects and morphisms, respectively, such that the following conditions hold:
\begin{itemize}
\item Any uncolored link component is either an annulus, or has both ends on the same uncolored coupon, in which case they are required to be adjacent to one another.
\item Uncolored coupons are all of the form shown in \firef{f:coupon}
\end{itemize}
\end{definition}
\begin{figure}[ht]
\figscale{.7}{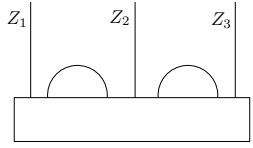}
\caption{A uncolored coupon $C$ in a special link $X$. Colored strands have a single end terminating on $C$. Uncolored strands have both ends terminating on $C$. Further, the ends are adjacent to one another.}
\label{f:coupon}
\end{figure}
\begin{theorem}
Let $\mathcal{M}$ be an extended 3-manifold as in \ocite{mine}. Then $\mathcal{C}$ may be obtained from $S^3$ via surgery along some special link $X \subset S^3$, where we define surgery along $X$ by 
\begin{equation}
M_{X} = M_{L} \backslash \displaystyle \bigcup_{i} T(C_i)
\end{equation}

Here, we do ordinary surgery along all annular link components $L$, giving $M_L$, and remove handlebodies $T(C_i)$ which are tubular neighborhoods of uncolored coupons $C_i$, as shown in \firef{f:handlebody}.
\end{theorem}
\begin{figure}[ht]
\figscale{.7}{coupon.eps} \hspace{1cm} $\longrightarrow$ \hspace{1cm} \figscale{.7}{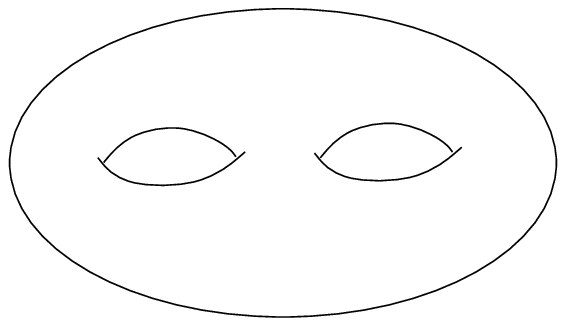}
\caption{A coupon $C$ determines a handlebody $T(C)$ by taking a tubular neighborhood of $C$ and its uncolored strands. The genus of $T(C)$ equals the number of uncolored strands incident to $C$. The colored strands determine extra data attached to this handlebody, namely marked points and tangent vectors on $\partial(T(C))$ (not pictured), and are important in definining Reshetikhin-Turaev theory. For more details, see \ocite{BK}.} 
\label{f:handlebody}
\end{figure}
Using the surgery description of an extended 3-manifold $\mathcal{M}$, we can easily define the RT invariant for such a manifold. Namely, we express $\mathcal{M}$ as the result of surgery along a \textit{special} link $X \subset S^3$, and define 
\begin{equation}
Z_{RT}(\mathcal{M}) \equiv Z_{RT}(X)
\end{equation}
where $Z_{RT}(X)$ is obtained by evaluating the special link $X$, summing over all possible colorings of unlabeled link components, and using the convention that whenever we color an unlabeled component by simply object $Y \in Z(\mathcal{C})$, we multiply by $d_Y$, its categorical dimension. As noted above, if $X$ has any uncolored coupons, then $Z_{RT}(X)$ is a vector, not a number.

We can also define Turaev-Viro theory on manifolds with embedded special links.

\begin{definition}
Let $\mathcal{M} = S_{X}^3$ be the 3-sphere with a special link $X$ inside. Then
\begin{equation}
Z_{TV, \mathcal{C}}(\mathcal{M}) \equiv Z_{TV, \mathcal{C}}(\mathcal{M'})
\end{equation}
where $\mathcal{M'}$ denotes the manifold with boundary obtained by removing tubular neighborhoods of each coupon.
\end{definition}
\begin{lemma} \label{l:tqfthandle}
Let $\mathcal{N}$ be a handlebody of genus $g$. Then $Z_{TV, \mathcal{C}}(\mathcal{N}) = Z_{RT, Z(\mathcal{C})}(\mathcal{N})$.
\end{lemma}
This equality is to be interpreted as follows: Under the canonical isomorphism $Z_{TV,\mathcal{C})}(\Sigma) \cong Z_{RT, Z(\mathcal{C})}(\Sigma)$, where $\Sigma = \partial \mathcal{N}$, the two sides of the equation are identified.
\begin{proof}
As computed in \ocite{BK} (Example 4.5.3), 
\begin{equation}
Z_{RT, Z(\mathcal{C})}(\mathcal{N}) = (id: \one \to (\one \otimes \one)^{\otimes g}) \in \displaystyle \bigoplus \Hom(1, Z_1 \otimes Z_{1}^* \otimes \dots \otimes Z_{g} \otimes Z_{g}^*).
\end{equation}. 
One shows that this equals $Z_{TV,\mathcal{C}}(\mathcal{N})$ by explicit state-sum computation. For $g =1$, this can be deduced from \ocite{mine2}, (Lemma 3.2). The general case is left to the reader as an exercise.
\end{proof}

Our goal is to show that $Z_{TV, \mathcal{C}}(\mathcal{M}) \cong Z_{RT,Z(\mathcal{C})}(\mathcal{M})$ for any extended 3-manifold $\mathcal{M}$. The idea is to convert our extended 3-manifold to a closed 3-manifold $\mathcal{N}$ (possibly with a link inside) by gluing handlebodies of appropriate genus to each component of $\partial \mathcal{M}$.

For simplicity, assume $\partial \mathcal{M}$ has a single component of genus $g$. Given a vector $\psi \in Z_{TV, \mathcal{C}}(\partial{\overline{M}})$, we can try to find a handlebody $\mathcal{H}_{g}$ with an embedded colored tangle, such that $Z_{TV,\mathcal{C}}(\mathcal{H}_{g})= \psi$. By the gluing axiom, we get
\begin{equation*}
\<Z_{TV,\mathcal{C}}(\mathcal{M}),\psi \> = \<Z_{RT,Z(\mathcal{C})}(\mathcal{M}),\psi \> 
\end{equation*} 

If we can do this for any $\psi$, then we are done. 

Unfortunately, it is almost never possible to produce such a handlebody, even when $g=0$. For example, if $\Sigma$ is the 3-punctured sphere with boundary components labeled by $Z_1, Z_2, Z_3$, the space $Z_{TV, \mathcal{C}}(\Sigma)$ may be quite large,  but there are no extended 3-manifolds with boundary $\Sigma$. Indeed $\Sigma$ is cobordant to $\emptyset$ if and only if it has an even number of punctures.

One approach is to redefine $Z_{TV,\mathcal{C}}(\Sigma)$ as the vector space generated by $\{Z_{TV,\mathcal{C}}(\mathcal{M})\}$, where $\mathcal{M}$ ranges over all extended 3-manifolds with $\partial \mathcal{M} = \Sigma$. The theorem then follows from the above argument and some minor details, which we omit. This approach certainly works, but it is undesirable, e.g. it assigns a zero-dimensional vector space to any surface with an odd number of boundary circles.

A better approach is to use coupons.
\begin{definition} \label{d:coupontv}
Let $\mathcal{K}$ be the 3-ball with an special link inside as shown in \firef{f:couponball}. We define 
\begin{equation}
Z_{TV, \mathcal{C}}(\mathcal{K}) = \psi \in Hom_{Z(\mathcal{C})}(\one, Z_{1} \otimes Z_{2} \otimes Z_{3} \otimes \dots \otimes Z_{N})
\end{equation}
\end{definition}
\begin{figure}[ht]
\figscale{.5}{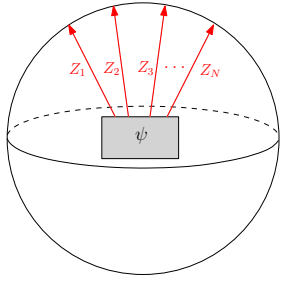}
\caption{The manifold $\mathcal{K}$ is the 3-ball $B^3$ with a special link consisting of a single coupon and $N$ strands connecting the coupon to $\partial K$. The coupon is labeled  by some morphism $\psi \in \Hom_{Z(\mathcal{C})}(\one, Z_1 \otimes \dots \otimes Z_N)$.}
\label{f:couponball}
\end{figure}

Notice that if we removed from $\mathcal{K}$ a tubular neighborhood of the coupon, we would be left with a cylinder over the $N$ punctured sphere. We think of the coupon as a handlebody $\mathcal{H}$ which satisfies $Z_{TV, \mathcal{C}}(\mathcal{H}) = \psi$.

Also note that we have defined the value of $Z_{TV,\mathcal{C}}(\mathcal{K})$. It is not a result we can deduce from standard Turaev-Viro theory. However, it is consistent with the rest of our theory. In particular, we can treat $\mathcal{K}$ as an ordinary extended 3-manifold and use the gluing axiom, the graphical calculus descibed in \ocite{mine}. In particular, if we view the braiding isomorphism, the cup and the cap (\ocite{mine2}, Section 2) as special examples of coupons, we get the same result as in \deref{d:coupontv}.
\begin{theorem} \label{t:special}
Let $\mathcal{M} =S_{X}^3$ be the 3-sphere with a special link $X$ inside. Then
\begin{equation}
Z_{TV,\mathcal{C}}(\mathcal{M}) = Z_{RT, Z(\mathcal{C})}(\mathcal{M})
\end{equation}
\end{theorem}
\begin{proof}
A special case of this theorem is proved in \ocite{mine2} (Theorem 2.3), where the result is proved if $X$ is a colored link (with no coupons). If every coupon of $X$ has the same form as that in \firef{f:couponball} (every strand that is incident to a coupon is colored and touches the coupon exactly once), the theorem follows immediately from Theorem 2.3 in \ocite{mine2}, \deref{d:coupontv} and the gluing axiom.

The situation is slightly more complicated if there are coupons of $X$ that have uncolored strands (see \firef{f:coupon}). We could try to come up with an analogous defintion to \deref{d:coupontv} for the more complicated coupons, but in this case, the result follows from the state sum formula and \deref{d:coupontv}.

Let $\mathcal{H}$ be the extended 3-manifold shown in \firef{f:cg}. $\mathcal{H}$ is a cobordism, so by standard theory, it gives a linear map
\begin{equation} \label{e:ident}
Z_{TV,\mathcal{C}}(\mathcal{H}): Z_{TV, \mathcal{C}}(S^2, Z_1, Z^{*}_1, Z_2, Z^{*}_2)  \longrightarrow  Z_{TV, \mathcal{C}}(\Sigma_{2})
\end{equation}

Notice that the left hand side of \eqref{e:ident} is naturally a subspace of $Z_{TV,\mathcal{C}}(\Sigma_{2})$.
\begin{lemma} \label{l:cobid}
The map defined above is given by
\begin{equation}
Z_{TV, \mathcal{C}}(\mathcal{H}) = Id:  Z_{TV, \mathcal{C}}(S^2, Z_1, Z^{*}_1, Z_2, Z^{*}_2) \longrightarrow Z_{TV,\mathcal{C}}(\Sigma_2)
\end{equation}
where by Id, we mean the identification of the domain with its image under the identity map. If we sum up over all possible colorings of the strands, the two sides of the equation are naturally isomorphic, and we get the identity map.
\end{lemma}
\begin{proof}
We decompose $\mathcal{H}$ as shown in \firef{f:coupondec}.

\begin{figure}[ht]
\figscale{1}{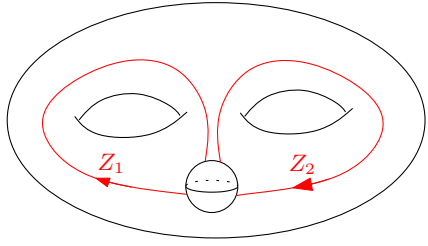}
\caption{The extended manifold $\mathcal{H}$ is obtained by taking a handlebody of genus 2, removing a 3-ball and embedding a colored tangle inside as shown. Thus, $\mathcal{H}$ is a cobordism between the sphere with 4 holes and $\Sigma_2$, a suface of genus 2.}
\label{f:cg}
\end{figure}

The result follows immediately from (\ocite{mine}, Example 9.2), where the computation is done in detail.
\end{proof}

An analogous result holds for a handlebody of any genus.

\begin{figure}[ht]
\figscale{1}{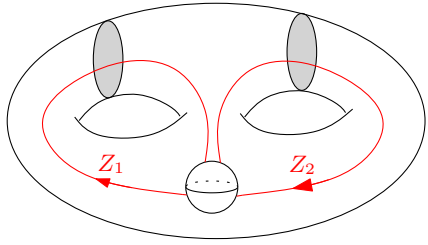}
\caption{The decomposition of $\mathcal{H}$ is given by cutting along the grey disks. The complement of these disks is a cylinder over the sphere with 4 holes. We choose the same cell decomposition (not pictures) for this cylinder as in \ocite{mine}, Example 9.2.}
\label{f:coupondec}
\end{figure}

Now we can use \leref{l:cobid} to finish proving \thref{t:special}. Suppose $X$ contains a coupon $C$ with uncolored strands beginning and ending on $C$ (see \firef{f:coupon} for an example.) Let T(C) be a tubular neighborhood of $C$, as described earlier. It is a handlebody of some genus $g$. By definition, 

\begin{equation*}
Z_{TV, \mathcal{C}}(S^3_{X}) = Z_{TV, \mathcal{C}}(S^3_{X} \backslash T(C)).
\end{equation*}

Using \leref{l:cobid}, we can glue an extended manifold $\mathcal{H}$  to $S^3_{X} \backslash T(C)$. (Here, we sum up over all colorings of strands inside $\mathcal{H}$.) Up to natural isomorphism $Z_{TV,\mathcal{C}}(\mathcal{H})$ is the identity map, so gluing it to $S^3_{X} \backslash T(C)$ does not change the value of $Z_{TV,\mathcal{C}}(S^3_{X} \backslash T(C))$.

 Our new manifold may be described as $S^{3}_{X'}$, where $X'$ is the same as $X$ except the coupon $C$ is replaced by a coupon $C$ with no uncolored strands incident to it (so $T(C')$ has genus zero). Repeating this, we reduce $X$ to a special link all of whose coupons have no uncolored strands incident to them. But we already know the theorem to be true in this case!
\end{proof}

We know from before that any extended 3-manifold may be obtained from $S^3_{X}$ by doing surgery along the annular components of $X$, and removing tubular neighborhoods of coupons of $X$. Combining \thref{t:special} with the surgery formula from \ocite{mine2} (Lemma 4.7) gives a proof of \thref{t:mainone}.

\section{Diagrams}\label{s:appendix}
\begin{landscape}
\begin{figure}[ht] 
\figscale{.8}{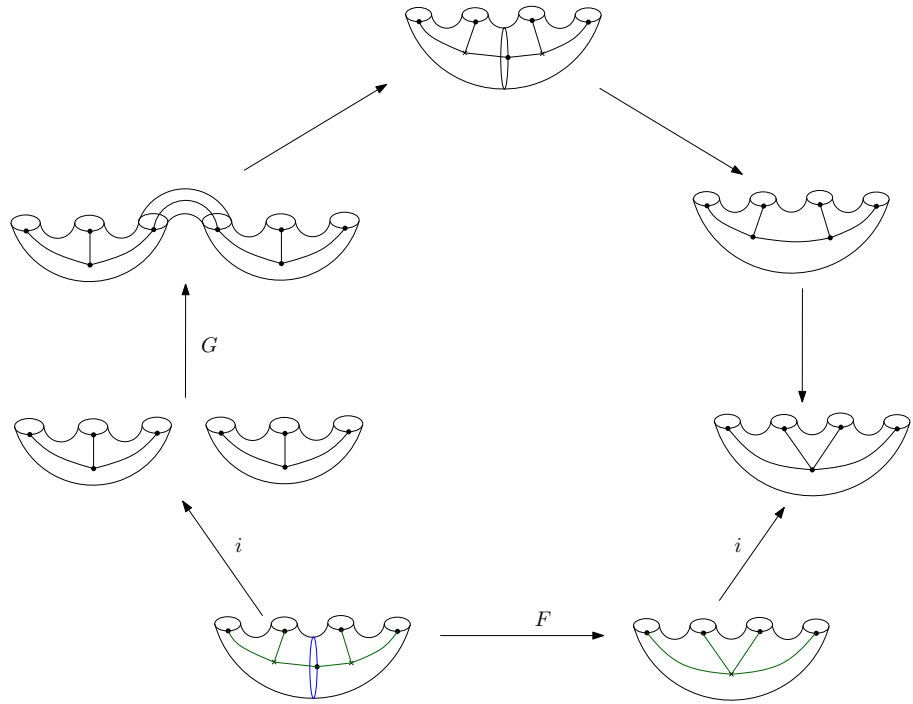}
\caption{The F-move fuses together two spheres along a boundary component. In terms of parametrized surfaces this is realized by simply removing a cut separating the spheres. At the level of cell decompositions we want to identify the result with the standard sphere decomposition (\firef{f:paramsphere}). We include several intermediate steps to make the computation more transparent.} \label{f:fusion}
\end{figure}
\end{landscape}
\begin{landscape}
\begin{figure} 
\figscale{1}{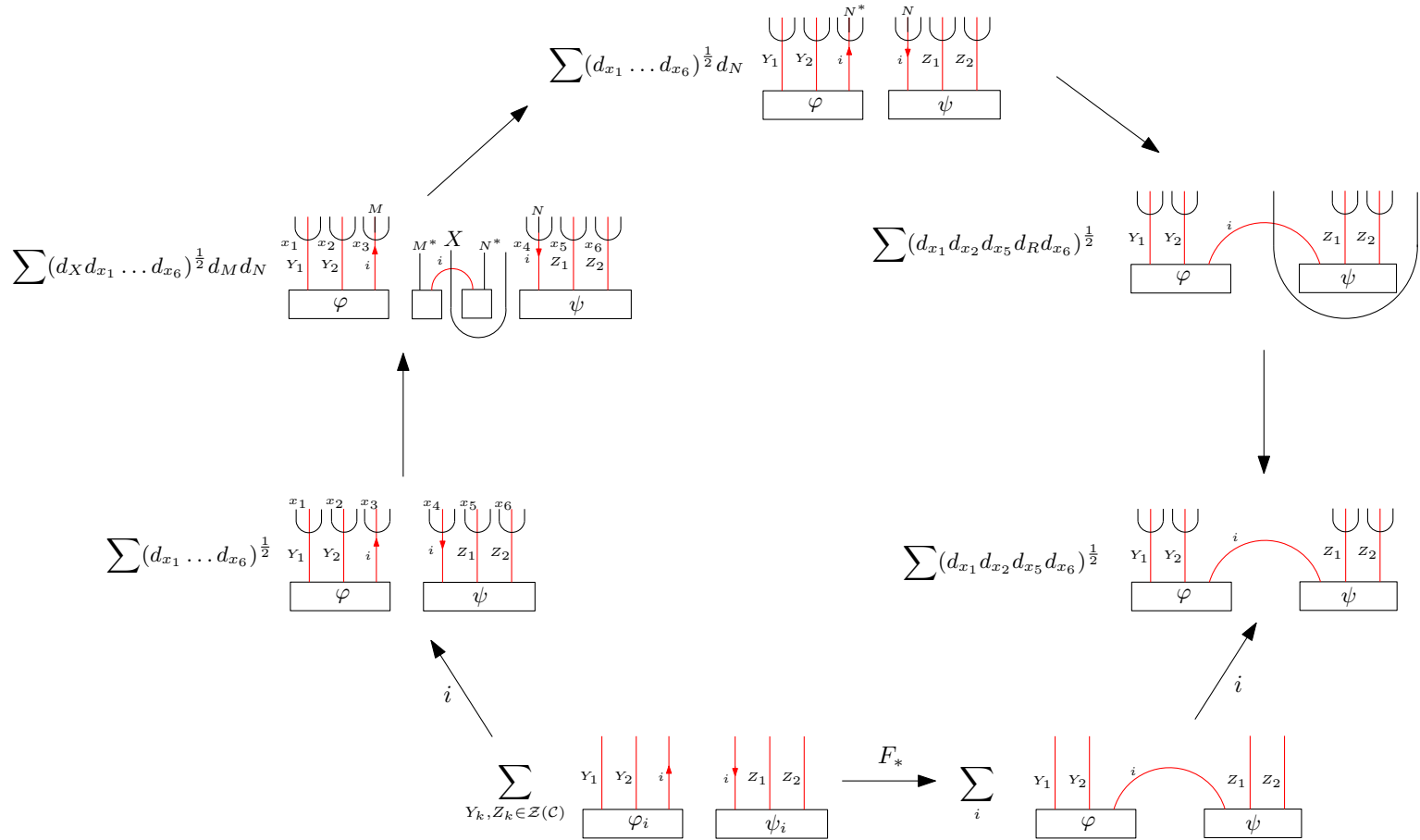}
\caption{A verification that \firef{f:fusion} commutes. All unlabeled arrow represent cylinder maps. We start on the lower left and proceed in two ways around the diagram.} \label{f:fvect}
\end{figure}
\end{landscape}
\begin{landscape}
\begin{figure}[ht] 
\fig{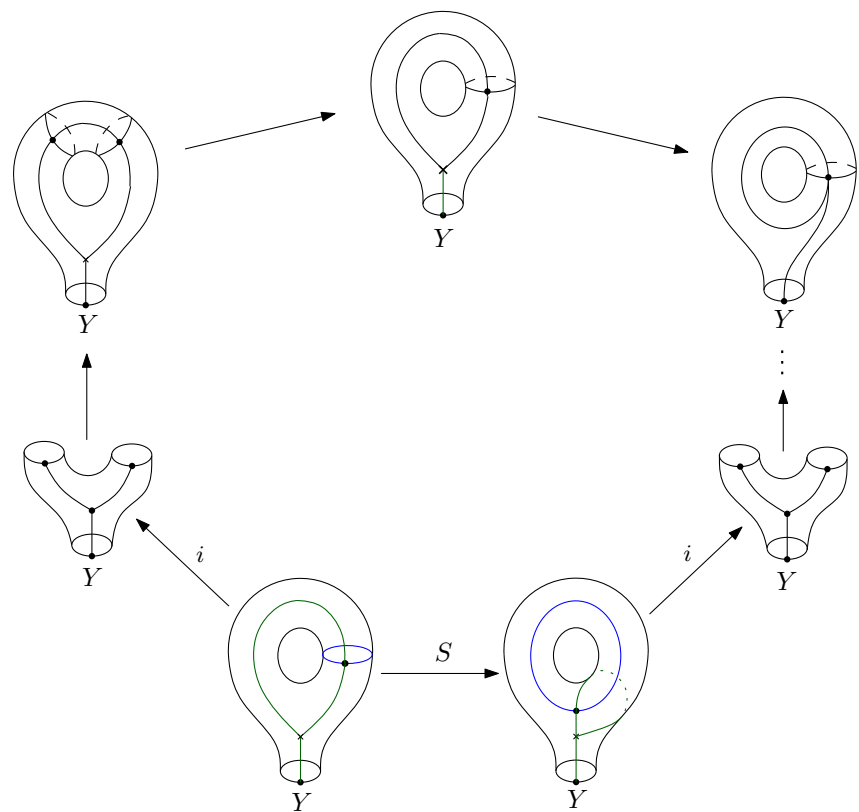}
\caption{The S-move interchanges meridians and longitudes of the 1-punctured torus. Thus, the cut (blue) and the parametrizing graph (green) exchange places under the application of S.} \label{f:sdiagram}
\end{figure}
\end{landscape}
\begin{landscape}
\begin{figure}[ht] 
\figscale{1.1}{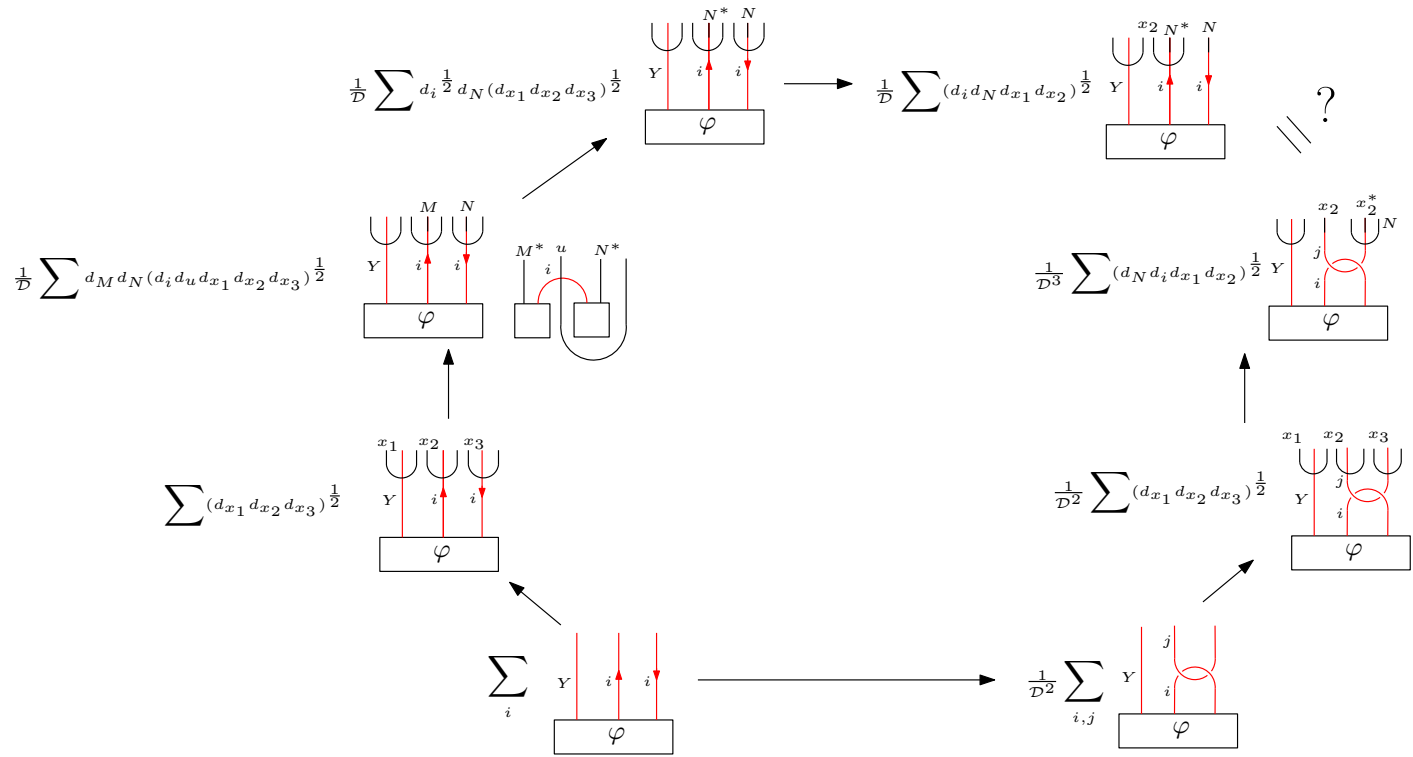}
\caption{A demonstration that \firef{f:sdiagram} commutes. Notice that when we proceed around the diagram we get two different pictures (separated by $"\stackrel{?}{=}"$), but can verify that they represent the same vector.} \label{f:svect}
\end{figure}
\end{landscape}

\end{document}